\documentclass[10pt]{article}
\usepackage{graphicx}
\usepackage{amssymb,latexsym,amscd,amsmath,amsfonts,amsthm,pb-diagram,enumerate}
\usepackage[mathscr]{eucal}
\usepackage{floatflt}
\numberwithin{equation}{section}
\newtheorem{theorem}{Theorem}[section]
\newtheorem{proposition}[theorem]{Proposition}
\newtheorem{corollary}[theorem]{Corollary}
\newtheorem{lemma}[theorem]{Lemma}
\newtheorem{question}[theorem]{Question}
\newtheorem{remark}[theorem]{Remark}
\newtheorem{notation}[theorem]{Notation}

\newtheorem{example}[theorem]{Example}
\newtheorem{definition}[theorem]{Definition}
\begin{document}
\begin{center} {\Large On the uniform bound of the index of reducibility of parameter ideals of a module whose polynomial type is at most one}
 \footnote [1]{
\indent{\bf AMS Classification 2010:} 13H10; 13D45.\\
This research is funded by Vietnam National Foundation for Science
and Technology Development (NAFOSTED) under grant number
101.01-2011.49.}
\end{center}
\begin{center}
                {PHAM HUNG QUY}
\end{center}
\begin{abstract}
Let $(R, \frak m)$ be a Noetherian local ring, $M$ a finitely
generated $R$-module. The aim of this paper is to prove a uniform
formula for the index of reducibility of paprameter ideals of $M$
provided the polynomial type of $M$ is at most one.
\end{abstract}
\section{Introduction}
Throughout this paper, let $(R, \frak m)$ be a Notherian local ring,
$M$ a finitely generated $R$-module of dimension $d$. Let
$\underline{x} = x_1,...,x_d$ be a system of parameters of $M$ and
$\frak q = (x_1,...,x_d)$. Let $\underline{n} = (n_1,...,n_d)$ be a
$d$-tuple of positive integers and $\underline{x}^{\underline{n}} =
x_1^{n_1},...,x_d^{n_d}$. We consider the difference
$$I_{M,\underline{x}}(\underline{n}) = \ell(M/(\underline{x}^{\underline{n}})M) -
e(\underline{x}^{\underline{n}};M)$$ as function in $\underline{n}$,
where $e(\underline{x};M)$ is the Serre multiplicity of $M$ with
respect to the sequence $\underline{x}$. Although
$I_{M,\underline{x}}(\underline{n})$ may be not a polynomial for
$n_1,...,n_d$ large enough, it is bounded above by polynomials.
Moreover, N.T. Cuong in \cite{C92} proved that the least degree of
all polynomials in $\underline{n}$ bounding above
$I_{M,\underline{x}}(\underline{n})$ is independent of the choice of
$\underline{x}$, and it is denoted by $p(M)$. The invariant $p(M)$
is called {\it the polynomial type} of $M$. Recalling that $M$ is a
Cohen-Macaulay module if and only if $\ell(M/\frak qM) = e(\frak
q;M)$ for some (and hence for all) parameter ideal $\frak q$ of $M$.
Thus, if we stipulate the degree of the zero polynomial is $-
\infty$, then $M$ is a Cohen-Macaulay module if and only if $p(M) =
- \infty$. In order to generalize the class of Cohen-Macaulay
module, J. Stuckrad and W. Vogel introduced the class of Buchsbaum
modules. An $R$-module $M$ is called Buchsbaum if and only if the
difference $\ell(M/\frak qM) - e(\frak q;M)$ is a constant for all
$\frak q$. For the theory of Buchsbaum modules see \cite{SV86}.
Furthermore, N.T. Cuong, P. Schenzel and N.V. Trung introduced the
class of generalized Cohen-Macaulay modules. Module $M$ is
generalized Cohen-Macaulay module if and only if the difference
$\ell(M/\frak qM) - e(\frak q;M)$ is bounded above for all parameter
ideals $\frak q$. In that paper they showed that $M$ is generalized
Cohen-Macaulay if and only if the $i$-th local cohomology module
$H^i_{\frak m}(M)$ has finite length for all $i = 0,...., d-1$. Set
$I(M) = \sup_{\frak q} \{\ell(M/\frak qM) - e(\frak q;M)\}$ where
the supremum is taken over all parameter ideals of $M$. If $M$ is a
generalized Cohen-Macaulay module we have $I(M) =
\sum_{i=0}^{d-1}\binom{d-1}{i} \ell(H^i_{\frak m}(M))$, and this
invariant is called the Buchsbaum invariant of $M$ (see
\cite{CST78}, \cite{Tr86}). It is easy to see that $M$ is a
generalized Cohen-Macaulay module if and only if $p(M) \leq 0$. The
structure of $M$ when $p(M) > 0$ is known little and there is no standard techniques to study since the local cohomology $H^i_{\frak m}(M)$ may be not
finitely generated for all $i \geq 1$. Even though the case $p(M) = 1$, the proof sometimes is very complicate (for example, see \cite{AGH05}).

Let $\frak q$ be a parameter ideal of $M$. The number of
irreducibility components appear in an irredundant irreducible
decomposition of $\frak qM$ is called the {\it index of
reducibility} of $\frak q$ on $M$, and denoted by
$\mathcal{N}_R(\frak q, M)$. It is well known that
$\mathcal{N}_R(\frak q, M) = \dim_{R/\frak m}\mathrm{Soc}(M/\frak
qM)$, where $\mathrm{Soc}(N) = 0:_N\frak m \cong
\mathrm{Hom}(R/\frak m, N)$ for an arbitrary $R$-module $N$. A
classical result of D.G. Northcott claimed that the index of
reducibility of parameter ideals on a Cohen-Macaulay module is an
invariant of the module. The converse of this result is not true,
the first counterexample is given by S. Endo and M. Narita in
\cite{EN64}. If $M$ is generalized Cohen-Macaulay, S. Goto and N.
Suzuki proved that $\mathcal{N}_R(\frak q, M)$ has an upper bound,
more precisely
$$\mathcal{N}_R(\frak q, M) \leq \sum_{i=0}^{d-1} \binom{d}{i} \ell(H^i_{\frak m}(M)) +
\dim_{R/\frak m}\mathrm{Soc}(H^d_{\frak m}(M)) \quad (\star)$$ for
all paprameter ideals $\frak q$ of $M$ (cf. \cite[Theorem
2.1]{GS84}). It is worthy to mention that if $M$ is Buchsbaum, Goto
and H. Sakurai in \cite{GS03} showed that the inequality $(\star)$
becomes an equality for all parameter ideals $\frak q$ contained in
a large enough power of $\frak m$. In \cite[Theorem 1.1]{CT08}, N.T.
Cuong and H.L. Truong considered Goto-Sakurai's result for
generalized Cohen-Macaulay modules. In fact, they proved that
$$\mathcal{N}_R(\frak q, M) = \sum_{i=0}^{d} \binom{d}{i} \dim_{R/\frak m}\mathrm{Soc}(H^i_{\frak m}(M))$$
for all $\frak q \subseteq \frak m^n$, $n \gg 0$. Recently, N.T.
Cuong and the author reproved this result based on the study of the
splitting of local cohomology (cf. \cite{CQ11}). A generalization of
Cuong-Truong's result can be found in \cite{Q12}.

The aim of this paper is to extend the result of Goto and Suzuki for
the class of modules of the polynomial type at most one. We show
that if $M$ is a finitely generated $R$-module such that $p(M) \leq
1$, then $\mathcal{N}_R(\frak q, M)$ is bounded above for all
parameter ideals $\frak q$ of $M$.

This paper is organized as follows. In Section $2$ we recall the
notions of the polynomial type of a module and the index of
reducibility. This paper is inspired by the uniform property of the
minimal number of generators of ideals in local rings of dimension
one (cf. \cite[Chapter 3]{S}) which we mention in Section 3. By Matlis' dual we obtain a similar result for Artinian modules of dimension one. Based on this result we can give the proof of the
main result by using the standard techniques of local cohomology in Section 4.

\section{Preliminaries}
We first recall the notion of {\it the polynomial type} of a module.
Let $(R, \frak m)$ be a Notherian local ring, $M$ a finitely
generated $R$-module of dimension $d$. Let $\underline{x} =
x_1,...,x_d$ be a system of parameters of $M$ and $\frak q =
(x_1,...,x_d)$. Let $\underline{n} = (n_1,...,n_d)$ be a $d$-tuple
of positive integers and $\underline{x}^{\underline{n}} =
x_1^{n_1},...,x_d^{n_d}$. We consider the difference
$$I_{M,\underline{x}}(\underline{n}) = \ell(M/(\underline{x}^{\underline{n}})M) -
e(\underline{x}^{\underline{n}};M)$$ as function in $\underline{n}$,
where $e(\underline{x};M)$ is the Serre multiplicity of $M$ with
respect to the sequence $\underline{x}$. N.T. Cuong in \cite[Theorem
2.3]{C92} showed that the least degree of all polynomials in
$\underline{n}$ bounding above $I_{M,\underline{x}}(\underline{n})$
is independent of the choice of $\underline{x}$.
\begin{definition}\rm
The least degree of all polynomials in $\underline{n}$ bounding
above $I_{M,\underline{x}}(\underline{n})$ is called {\it the
polynomial type} of $M$, and denoted by $p(M)$.
\end{definition}
The following basic properties of $p(M)$ can be found in \cite{C92}.
\begin{remark}\rm
\begin{enumerate}[{(i)}]
\item We have $p(M) \leq d-1$.
\item  An $R$-module $M$ is Cohen-Macaulay if and only if $p(M) = -\infty$.
Moreover, $M$ is generalized Cohen-Macaulay if and only if $p(M) =
0$.
\item If we denote the $\frak m$-adic completion of $M$ by
$\widehat{M}$, then $p(M) = p_{\widehat{R}}(\widehat{M})$.
\end{enumerate}
\end{remark}
Let $a_i(M) = \mathrm{Ann}H^i_{\frak m}(M)$ for $0 \leq i \leq d-1$
and $\frak a(M) = \frak a_0(M) \cdots \frak a_{d-1}(M)$. We denote
by $NC(M)$ the non-Cohen-Macaulay locus of $M$ i.e. $NC(M) = \{\frak
p \in \mathrm{supp}(M)\,|\, M_{\frak p} \,\, \text{is not
Cohen-Macaulay} \}$. Recalling that $M$ is called {\it
equidimensional} if $\dim M = \dim R/\frak p$ for all minimal
associated primes of $M$. The following result give the meaning of
the polynomial type.
\begin{theorem}[\cite{C91}, Theorem 1.2] \label{T2.3} Suppose that $R$ admits a dualizing complex. Then
\begin{enumerate}[{(i)}]\rm
\item $p(M) = \dim R/\frak a(M)$.
\item {\it If $M$ is equidimensional then $p(M) = \dim (NC(M))$.}
\end{enumerate}
\end{theorem}
\begin{example}\rm
Let $S = k[[X_1,X_2,...,X_{2n+1}]]/(X_1,...,X_n) \cap (X_{n+1},...,X_{2n})$ where $k$ is a field and $n$ is a positive integer greater then $1$. It is easy to see that $R$ is not generalized Cohen-Macaulay but $p(M) = 1$.
\end{example}
\begin{remark}\rm
By \cite{N62}, there exists a local domain $(R,\frak m)$ of dimension two such that the $\frak m$-adic completion of $R$ is $\widehat{R} = k[[X, Y, Z]]/(X) \cap (Y, Z)$. We can check that $\dim R/\frak a(R) = 2$ and $\dim (NC(R)) = 0$ but $\dim \widehat{R}/\frak a(\widehat{R}) =  \dim (NC(\widehat{R})) = p(\widehat{R}) = 1$. So $\dim R/\frak a(R)$ and $\dim (NC(R))$ may be change after passing to the completion. This is the reason we use the notion of the polynomial type in this paper.
\end{remark}
We next recall the object of the present paper.
\begin{definition}\rm
Let $\frak q$ be a parameter ideal of $M$. The {\it index of
reducibility} of $\frak q$ on $M$ is the number of irreducibility
components appear in an irredundant irreducible decomposition of
$\frak qM$, and denoted by $\mathcal{N}_R(\frak q, M)$.
\end{definition}
\begin{remark}\rm
\begin{enumerate}[{(i)}]
\item It is well known that $\mathcal{N}_R(\frak q, M) = \dim_{R/\frak
m}\mathrm{Soc}(M/\frak qM)$, where $\mathrm{Soc}(N) = 0:_N\frak m
\cong \mathrm{Hom}(R/\frak m, N)$ for an arbitrary $R$-module $N$.
\item If $M$ is Cohen-Macaulay i.e. $p(M) = -\infty$, then $\mathcal{N}_R(\frak q, M) = \dim_{R/\frak
m}\mathrm{Soc}(H^d_{\frak m}(M))$ for all parameter ideals $\frak
q$.
\item If $M$ is generalized Cohen-Macaulay i.e. $p(M) = 0$, then Goto and Suzuki proved that
$$\mathcal{N}_R(\frak q, M) \leq \sum_{i=0}^{d-1} \binom{d}{i} \ell(H^i_{\frak m}(M)) +
\dim_{R/\frak m}\mathrm{Soc}(H^d_{\frak m}(M))$$ for all paprameter
ideals $\frak q$ of $M$ (cf. \cite[Theorem 2.1]{GS84}). Furthermore,
let $n_0$ be a positive integer such that $\frak m^{n_0}H^i_{\frak
m}(M) = 0$ for all $i = 0,...,d-1$. In \cite[Corollary 4.3]{CQ11}
N.T. Cuong and the author showed that for all parameter ideal $\frak
q$ contained in $\frak m^{2n_0}$ we have
$$\mathcal{N}_R(\frak q, M) = \sum_{i=0}^{d} \binom{d}{i} \dim_{R/\frak m}\mathrm{Soc}(H^i_{\frak m}(M)).$$
\end{enumerate}
\end{remark}
\section{Modules of dimension one}
Notice that $p(M)$ and $\mathcal{N}_R(\frak q, M)$ do not change
after passing to the $\frak m$-adic completion. Therefore, in the
rest of this paper we always assume that $(R, \frak m)$ is a
complete ring. In this paper we consider the boundness of $\mathcal{N}_R(\frak q,
M)$ provided $p(M) \leq 1$. In this case Theorem \ref{T2.3} implies
that $H^i_{\frak m}(M)$ is Artinian with $\dim
R/\mathrm{Ann}H^i_{\frak m}(M) \leq 1$ for all $i=0,...,d-1$. Hence
the Matlis' dual of $H^i_{\frak m}(M)$ is a Noetherian module of
dimension at most one for all $i=0,...,d-1$. The minimal number of
generators of a module $N$ will be denoted by $v(N)$. The key role
in our proof of the main result is the following interesting result
of local ring of dimension one (see \cite[Chapter 3]{S}).
\begin{lemma}
Let $(R, \frak m)$ be a local ring of dimension one. Then the
minimal number of generators of ideals of $R$ is bounded above by an
invariant independent of the choice of ideals i.e there is a
positive integer $c$ such that $v(I) = \ell (I/\frak mI) \leq c$ for
all ideal $I$.
\end{lemma}
Goto and Suzuki in \cite[Theorem 3.1]{GS84} extended above result
for modules as follows.
\begin{lemma}\label{L3.2}
Let $M$ be a finitely generated $R$-module of dimension one. Then
there is a positive integer $c$ such that $v(N) = \ell (N/\frak mN)
\leq c$ for all submodule $N$ of $M$.
\end{lemma}
\begin{notation}\rm Let $M$ be a finitely generated $R$-module. We
define
$$c(M) = \sup_N \{v(N )\, |\, N \subseteq M\}.$$
\end{notation}
\begin{remark}\rm
\begin{enumerate}[{(i)}]
\item By Lemma \ref{L3.2} we have if $d \leq 1$, then $c(M)$ is a
positive integer. Moreover if $d = 0$ then $c(M) \leq \ell(M)$.
\item If $d \geq 2$ since $v(\frak m^nM)$ is a polynomial of degree
$d-1$ when $n \gg 0$, then $c(M) = \infty$.
\end{enumerate}
\end{remark}
We present some properties of the invariant $c(M)$.
\begin{proposition} Let $M$ be a finitely generated $R$-module of
dimension $d \leq 1$. Then
$$c(M) = \sup_N \{\ell(N:_M \frak m/N)\, |\, N \subseteq M\}.$$
\end{proposition}
\begin{proof}
The assertion follows form the facts
$$\ell(N:_M \frak m/N) \leq v(N:_M \frak m),$$
and
$$v(N) \leq \ell ((\frak mN :_M \frak m)/\frak mN).$$
\end{proof}

\begin{proposition}\label{P3.6}
We consider the following short exact sequence of finitely generated
$R$-modules of dimension at most one
$$0 \to M_1 \to M \to M_2 \to 0.$$
Then
\begin{enumerate}[{(i)}]\rm
\item {\it $c(M_1) \leq c(M)$ and $c(M_2) \leq c(M)$.}
\item $c(M) \leq c(M_1) + c(M_2)$.
\end{enumerate}
\end{proposition}
\begin{proof}
(i) immediately follows from the definition of $c(M)$.\\
(ii) Let $N$ be a submodule of $M$ such that $v(N) = c(M)$. There
are submodules $N_1$ and $N_2$ of $M_1$ and $M_2$, respectively,
such that
$$0 \to N_1 \to N \to N_2 \to 0$$
is a short exact sequence. Then
$$c(M) =v(N) \leq v(N_1) + v(N_2) \leq c(M_1) + c(M_2).$$
\end{proof}

\section{The main result}
Recalling that we always assume that $(R, \frak m)$ be a complete
Notherian local ring. Let $E(R/\frak m)$ be the injective hull of
$R$-module $R/\frak m$. Let $A$ be an Artinian $R$-module. We have
the Matlis' dual of $A$, $N = \mathrm{Hom}(A, E(R/\frak m))$, is
Noetherian and $\mathrm{Ann}A = \mathrm{Ann} N$. In this section we
say an Artinian $R$-module $A$ of dimension $t$ if its dual is a
Noetherian module of dimension $t$, i.e. $\dim R/\mathrm{Ann}A = t$.
It is well known that $H^i_{\frak m}(M)$ is Artinian for all $i \geq
0$ (see \cite[Chapter 7]{BS}). Theorem \ref{T2.3} claims that a
finitely generated $R$-module $M$ of dimension $d$ and $p(M) \leq 1$
if and only if $\dim H^i_{\frak m}(M) \leq 1$ for all $i = 0,...,
d-1$. For the study of dimension of an Artinian module on a general
Noetherian local ring see \cite{CN02}. We need the following lemma.

\begin{lemma}[\cite{BS}, Lemma 10.2.16] \label{L3.1}
Let $E, F, I$ be $R$-module such that $E$ is finitely generated and
$I$ is injective. Then
$$\mathrm{Hom}(\mathrm{Hom}(E, F), I) \cong E \otimes \mathrm{Hom}(F, I).$$
\end{lemma}
For an Artinian $R$-module $A$ we set $r(A):= \sup \{\ell(B:_A \frak
m/B)\, |\, B \subseteq A \}$.
\begin{corollary}\label{C4.2} Let $A$ be an Artinian $R$-module of dimension at
most one. Let $N = \mathrm{Hom}(A, E(R/\frak m))$. Then $r(A) =
c(N)$.
\end{corollary}
\begin{proof} For each
submodule $B$ of $A$ set $L = \mathrm{Hom}(A/B, E(R/\frak m))$, then
$L$ is a submodule of $N$. By Lemma \ref{L3.1} we have
$$\mathrm{Hom}(\mathrm{Hom}(R/\frak m, A/B), E(R/\frak m)) \cong R/\frak m \otimes \mathrm{Hom}(A/B, E(R/\frak m)).$$
Hence $\ell(B:_A \frak m/B) = v(L) \leq c(N)$. Thus $r(A) \leq
c(N)$. Conversely, let $L$ be a submodule of $N$ such that $v(L) =
c(N)$. Let $B = \mathrm{Hom}(N/L, E(R/\frak m))$. We have $B$ is a
submodule of $A$. By duality we have $N/L \cong \mathrm{Hom}(B,
E(R/\frak m))$ so $L \cong \mathrm{Hom}(A/B, E(R/\frak m))$. As
above we have $\ell(B:_A \frak m/B) = v(L) = c(N)$, so $r(A) \geq
c(N)$.
\end{proof}

The next result immediately follows from Corollary \ref{C4.2} and
Proposition \ref{P3.6}.
\begin{corollary}\label{C4.3}
We consider the following short exact sequence of Artinian
$R$-modules of dimension at most one
$$0 \to A_1 \to A \to A_2 \to 0.$$
Then
\begin{enumerate}[{(i)}]\rm
\item {\it $r(A_1) \leq r(A)$ and $r(A_2) \leq r(A)$.}
\item $r(A) \leq r(A_1) + r(A_2)$.
\end{enumerate}
\end{corollary}
Recalling that a sequence of elements $x_1,...,x_k$ is called a {\it
filter regular sequence} of $M$ if $\mathrm{Supp}\,
((x_1,...,x_{i-1})M:x_i)/(x_1,...,x_{i-1})M \subseteq \{\frak m\}$
for all $i = 1,...,k$.
\begin{proposition}\label{P4.4}
Let $M$ be a finitely generated $R$-module of dimension $d$ and
$p(M) \leq 1$. Then for every filter regular sequence $x_1,...,x_k$,
$k \leq d$, of $M$ we have
$$r(H^j_{\frak m}(M/(x_1,...,x_k)M)) \leq \sum_{i=j}^{j+k} \binom{k}{i-j} r(H^i_{\frak m}(M))$$
for all $j < d-k$, and
$$\dim \mathrm{Soc}(H^{d-k}_{\frak m}(M/(x_1,...,x_k)M)) \leq \sum_{i=d-k}^{d-1} \binom{k}{k+i-d} r(H^i_{\frak
m}(M)) + \dim \mathrm{Soc}(H^{d}_{\frak m}(M)).$$
\end{proposition}
\begin{proof} Induction on $k$, the case $k=0$ is trivial. If $k= 1$, the
short exact sequence
$$0 \to M/0:_Mx_1 \to M \to M/x_1M \to 0$$
induces the following exact sequence
$$H^j_{\frak m}(M) \to H^j_{\frak m}(M/x_1M) \to H^{j+1}_{\frak m}(M/0:_Mx_1)$$
for all $j < d-1$. Since $\ell(0:_Mx_1) < \infty$ we have
$H^{j+1}_{\frak m}(M/0:_Mx_1) \cong H^{j+1}_{\frak m}(M)$ for all $j
\geq 0$. Hence $\mathrm{Ann} H^j_{\frak m}(M/x_1M) \supseteq
\mathrm{Ann}H^j_{\frak m}(M) \mathrm{Ann}H^{j+1}_{\frak m}(M)$ for
all $j < d-1$. Therefore $\dim R/\mathrm{Ann} H^j_{\frak m}(M/x_1M)
\leq 1$ for all $j = 0,...,d-2$, so $p(M/x_1M) \leq 1$. By Corollary
\ref{C4.3} it is clear that
$$r(H^j_{\frak m}(M/x_1M)) \leq r(H^j_{\frak m}(M)) + r(H^{j+1}_{\frak m}(M))$$
for all $j < d-1$. On the other hand we have the following exact
sequence
$$H^{d-1}_{\frak m}(M) \to H^{d-1}_{\frak m}(M/x_1M) \to H^{d}_{\frak m}(M) \overset{x}{\to} H^{d}_{\frak m}(M).$$
Thus we have the short exact sequence
$$0 \to A \to H^{d-1}_{\frak m}(M/x_1M) \to 0:_{H^{d}_{\frak m}(M)}x_1 \to 0,$$
where $A$ is a quotient of $H^{d-1}_{\frak m}(M)$. By applying the
functor $\mathrm{Hom}(R/\frak m, \bullet)$ to the above short exact
sequence we get the following exact sequence
$$0 \to 0:_A\frak m \to 0:_{H^{d-1}_{\frak m}(M/x_1M)}\frak m \to 0:_{H^{d}_{\frak m}(M)}\frak m.$$
Therefore
\begin{eqnarray*}
\dim \mathrm{Soc}(H^{d-1}_{\frak m}(M/x_1M)) &\leq& \dim
\mathrm{Soc}(A) + \dim \mathrm{Soc}(H^{d}_{\frak m}(M))\\
&\leq& r(A) + \dim \mathrm{Soc}(H^{d}_{\frak m}(M))\\
&\leq& r(H^{d-1}_{\frak m}(M)) + \dim \mathrm{Soc}(H^{d}_{\frak
m}(M)).
\end{eqnarray*}
So the assertion holds true if $k=1$. For $k> 1$ by induction we
have
\begin{eqnarray*}
r(H^j_{\frak m}(M/(x_1,...,x_k)M)) &\leq & r(H^j_{\frak
m}(M/(x_1,...,x_{k-1})M)) + r(H^{j+1}_{\frak
m}(M/(x_1,...,x_{k-1})M))\\
& \leq & \sum_{i = j}^{k+j-1}\binom{k-1}{i-j}r(H^i_{\frak m}(M)) +
\sum_{i = j+1}^{k+j}\binom{k-1}{i-j-1}r(H^i_{\frak m}(M))\\
&=& \sum_{i=j}^{j+k} \binom{k}{i-j} r(H^i_{\frak m}(M))
\end{eqnarray*}
for all $j<d-k$. Moreover, we have\\
\newline
$ \dim \mathrm{Soc}(H^{d-k}_{\frak m}(M/(x_1,...,x_k)M))$
\begin{eqnarray*}
 & \leq & r(H^{d-k}_{\frak m}(M/(x_1,...,x_{k-1})M)) + \dim \mathrm{Soc}(H^{d-k+1}_{\frak m}(M/(x_1,...,x_{k-1})M))\\
&\leq & \sum_{k = d-k}^{d-1}\binom{k-1}{k+i-d}r(H^i_{\frak m}(M)) +
\sum_{i=d-k+1}^{d-1} \binom{k-1}{k+i-d-1} r(H^i_{\frak m}(M)) + \dim
\mathrm{Soc}(H^{d}_{\frak m}(M))\\
&=& \sum_{i=d-k}^{d-1} \binom{k}{k+i-d} r(H^i_{\frak m}(M)) + \dim
\mathrm{Soc}(H^{d}_{\frak m}(M)).
\end{eqnarray*}
The proof is complete.
\end{proof}
\begin{remark}\rm \label{R4.6}
It should be noted that for every parameter ideal $\frak q = (x_1,...,x_d)$ of
$M$ we can choose a system of parameters $\underline{y} = y_1,...,y_d$ which is a filter regular sequence such that $\frak q = (y_1,...,y_d)$. Indeed, by the prime avoidance theorem we can choose an element $y_1 \in \frak q \setminus \frak {mq}$ and $y_1 \notin  \frak p$ for all  $\frak p \in \mathrm{Ass}M$ and $\frak p \neq \frak m$. Therefore $y_1$ is both a parameter element and a filter regular element of $M$. For $i = 2,..., d$, by applying the prime avoidance theorem again there exists an element $y_i \in \frak q \setminus \big(\frak {mq}\cup (y_1,...,y_{i-1})\big)$ and $y_i \notin  \frak p$ for all  $\frak p \in \mathrm{Ass}M/(y_1,...,y_{i-1})M$ and $\frak p \neq \frak m$. Thus we have a system of parameters $\underline{y} = y_1,...,y_d$ which is also a filter regular sequence of $M$. The claim $\frak q = (y_1,...,y_d)$ is easy by the Nakayama lemma.
\end{remark}
Applying for $k=d$ in Proposition \ref{P4.4} and using Remark \ref{R4.6} we have the main result
of this paper as follows.
\begin{theorem}
Let $M$ be a finitely generated $R$-module of dimension $d$ and
$p(M) \leq 1$. Then the index of reducibility of parameter ideal
$\frak q$ of $M$ is bounded above by a invariant independent of the
choice of $\frak q$. Namely
$$\mathcal{N}_R(\frak q, M) \leq \sum_{i=0}^{d-1} \binom{d}{i} r(H^i_{\frak
m}(M)) + \dim \mathrm{Soc}(H^{d}_{\frak m}(M))$$ for all parameter
ideals $\frak q$ of $M$.
\end{theorem}
Notice that in the case $M$ is generalized Cohen-Macaulay our bound
is a sharp of the Goto-Suzuki one since $r(H^i_{\frak m}(M)) \leq
\ell(H^i_{\frak m}(M))$ for all $i$. An $R$-module $M$ is called
{\it unmixed} if $\dim R/\frak p = \dim M$ for all $\frak p \in
\mathrm{Ass}M$. It is not difficult to see that if $M$ is an unmixed
module of dimension three, then $p(M) \leq 1$ (see \cite[Theorem
8.1.1]{BH98}). The following is a generalization of \cite[Corollary
3.7]{GS84} for modules.
\begin{corollary}\label{C4.6} Let $M$ is an unmixed
module of dimension three. Then the index of reducibility of
parameter ideal $\frak q$ of $M$ is bounded above by a invariant
independent of the choice of $\frak q$.
\end{corollary}
If $(R, \frak m)$ is a local ring of dimension three, then
$\mathcal{N}_R(\frak q, R)$ is bounded above for all parameter
ideals $\frak q$ of $R$ (cf. \cite[Theorem 3.8]{GS84}). Thus the
class of modules for which the index of reducibility of parameter
ideals is bounded is strictly larger than the class of modules of
the polynomial type at most one. In \cite[Example 3.9]{GS84} Goto
and Suzuki also constructed a ring of dimension four and the index
of reducibility of parameter ideals are not bounded above. Notice
that the ring of Goto and Suzuki has the polynomial type three.
Therefore it is natural to raise the following question.\\
\begin{question}\rm Is it true that $\mathcal{N}_R(\frak q, M)$ is
bounded above for all parameter ideals $\frak q$ of $M$ if and only
if $p(M) \leq 2$.
\end{question}

{\bf Acknowledgements:} The author is grateful to Professor Nguyen
Tu Cuong for his valuable comments on the first draft of this paper.
This paper has been written while the author is visiting the Vietnam
Institute for Advanced Study in Mathematics (VIASM), Hanoi, Vietnam.
He would like to thank VIASM for their support and hospitality.

\textsc{Department of Mathematics, FPT University, 8 Ton That Thuyet Road, Ha Noi, Viet Nam}\\
 {\it E-mail address}: quyph@fpt.edu.vn

\begin{thebibliography}{99}
\bibitem{AGH05}  I.M. Aberbach, L. Ghezzi, H.T. Ha, Homology multipliers and the relation type of parameter ideals, {\it Pacific J. Math.} 226 (2006), 1--40.

\bibitem{BS} M. Brodmann, R. Y. Sharp, {\it Local cohomology: An algebraic introduction with geometric
applications,} Cambridge University Press, 1998.

\bibitem{BH98} W. Bruns, J. Herzog, {\it Cohen-Macaulay rings}, Cambridge University
Press (Revised edition), 1998.

\bibitem{C91} N.T. Cuong, On the dimension of the non-Cohen-Macaulay locus of local rings
admitting dualizing complexes, {\it Math. Proc. Cambridge Phil.
Soc.} 109 (1991), 479--488.

\bibitem{C92} N.T. Cuong, On the
least degree of polynomials bounding above the differences between
lengths and multiplicities of certain systems of parameters in local
ring, {\it Nagoya Math. J.} 125 (1992), 105--114.

\bibitem{CN02} N.T. Cuong, L.T. Nhan, On the Noetherian dimension of Artinian
modules, {\it Vietnam J. Maths.} 30 (2002), 121--130.

\bibitem{CQ11} N.T. Cuong, P.H. Quy, A splitting theorem for local cohomology and its applications, {\it J. Algebra} 331 (2011),
512--522.

\bibitem{CST78} N.T. Cuong, P. Schenzel, N.V. Trung,
Verallgeminerte Cohen-Macaulay moduln, {\it Math-Nachr.} 85 (1978),
156--177.

\bibitem{CT08} N.T. Cuong, H.L. Truong, Asymptotic behavior of parameter ideals in generalized Cohen-Macaulay
module, {\it J. Algebra} 320 (2008), 158--168.

\bibitem{EN64} S. Endo, M. Narita, The number of irreducible components of an ideal and the semi-regularity of a local ring,
{\it Proc. Japan Acad.} 40 (1964), 627--630.

\bibitem{GS03} S. Goto, H. Sakurai, The equality $I^2 = QI$ in Buchsbaum rings, {\it Rend. Sem. Univ. Padova.}
110 (2003), 25--56.

\bibitem{GS84} S. Goto, N. Suzuki, Index of reducibility of parameter ideals in a local ring, {\it J. Algebra} 87 (1984),
53--88.

\bibitem{N62} M. Nagata, {\it Local rings}, Interscience, New York,
1962.

\bibitem{Q12} P.H. Quy, Asymptotic behaviour of good systems of parameters of
sequentially generalized Cohen-Macaulay modules, {\it Kodai Math.
J.} 35 (2012), 576--588.

\bibitem{S} J.D. Sally, {\it Numbers of generators of ideals in local rings}, Marcel
Dekker, Inc., New York - Basel, 1978.

\bibitem{SV86} J. Stuckrad, W. Vogel, {\it Buchsbaum rings and applications}, Spinger-Verlag, 1986.

\bibitem{Tr86} N.V. Trung, Toward a theory of generalized Cohen-Macaulay modules, {\it Nagoya Math. J.} 102 (1986), 1--49.

\end{thebibliography}
\end{document}